\documentclass[10pt,a4paper]{scrartcl}

\usepackage[british]{babel}

\usepackage[a4paper,top=3cm,bottom=3cm,left=3cm,right=3cm]{geometry}

\usepackage{graphicx} 
\usepackage{amsmath}
\usepackage{amssymb}
\usepackage{amsthm}
\usepackage{stmaryrd}

\makeatletter
\DeclareRobustCommand{\em}{%
	\@nomath\em \if b\expandafter\@car\f@series\@nil
	\normalfont \else \slshape \fi}
\makeatother

\usepackage[utf8]{inputenc}
\usepackage{color}
\usepackage{enumitem}
\usepackage[affil-it]{authblk}
\usepackage{graphicx} 

\usepackage{tikz-cd}
\tikzstyle{tikzfig}=[baseline=-0.25em,scale=0.5]

\pgfkeys{/tikz/tikzit fill/.initial=0}
\pgfkeys{/tikz/tikzit draw/.initial=0}
\pgfkeys{/tikz/tikzit shape/.initial=0}
\pgfkeys{/tikz/tikzit category/.initial=0}

\pgfdeclarelayer{edgelayer}
\pgfdeclarelayer{nodelayer}
\pgfsetlayers{background,edgelayer,nodelayer,main}

\tikzstyle{none}=[inner sep=0mm]

\newcommand{\tikzfig}[1]{%
	{\tikzstyle{every picture}=[tikzfig]
		\IfFileExists{#1.tikz}
		{\input{#1.tikz}}
		{%
			\IfFileExists{./figures/#1.tikz}
			{\input{./figures/#1.tikz}}
			{\tikz[baseline=-0.5em]{\node[draw=red,font=\color{red},fill=red!10!white] {\textit{#1}};}}%
	}}%
}

\tikzstyle{every loop}=[]

\usepackage{tikzit}

\tikzstyle{blue}=[-, draw=blue]
\tikzstyle{dashed}=[-,dotted,draw=black]
\tikzstyle{dashedred}=[-,dotted,dashed,draw=red]
\tikzstyle{smallblue}=[fill=blue, draw=blue, shape=circle, minimum size=4pt, inner sep=0pt]
\tikzstyle{bluecircle}=[fill=white, draw=blue, shape=circle]
\tikzstyle{endarrow}=[->]

\usetikzlibrary{decorations.pathreplacing,decorations.markings}
\tikzset{
	on each segment/.style={
		decorate,
		decoration={
			show path construction,
			moveto code={},
			lineto code={
				\path [#1]
				(\tikzinputsegmentfirst) -- (\tikzinputsegmentlast);
			},
			curveto code={
				\path [#1] (\tikzinputsegmentfirst)
				.. controls
				(\tikzinputsegmentsupporta) and (\tikzinputsegmentsupportb)
				..
				(\tikzinputsegmentlast);
			},
			closepath code={
				\path [#1]
				(\tikzinputsegmentfirst) -- (\tikzinputsegmentlast);
			},
		},
	},
	mid arrow/.style={postaction={decorate,decoration={
				markings,
				mark=at position .7 with {\arrow[#1]{stealth}}
	}}},
}
\tikzset{%
	link/.style    = { white, double = black, line width = 1.8pt,
		double distance = 0.4pt },
	channel/.style = { white, double = black, line width = 0.8pt,
		double distance = 0.8pt },
}

\tikzset{%
	blink/.style    = { white, double = blue, line width = 2pt,
		double distance = 1pt },
	channel/.style = { white, double = blue, line width = 2pt,
		double distance = 1pt },
}

\tikzstyle{tikzfig}=[baseline=-0.25em,scale=0.5]

\pgfkeys{/tikz/tikzit fill/.initial=0}
\pgfkeys{/tikz/tikzit draw/.initial=0}
\pgfkeys{/tikz/tikzit shape/.initial=0}
\pgfkeys{/tikz/tikzit category/.initial=0}

\pgfdeclarelayer{edgelayer}
\pgfdeclarelayer{nodelayer}
\pgfsetlayers{background,edgelayer,nodelayer,main}

\tikzstyle{none}=[inner sep=0mm]

\tikzstyle{every loop}=[]

\newtheoremstyle{dtheorem}
{\topsep}
{\topsep}
{\slshape}
{0pt}
{\bfseries}
{.}
{ }
{\thmname{#1}\thmnumber{ #2}\thmnote{ {\normalfont\slshape(#3)}}}

\newtheoremstyle{ddefinition}
{\topsep}
{\topsep}
{\normalfont}
{0pt}
{\bfseries}
{.}
{ }
{\thmname{#1}\thmnumber{ #2}\thmnote{ {\normalfont\slshape(#3)}}}

\theoremstyle{dtheorem}
\newtheorem{theorem}{Theorem}[section]

\makeatletter
\newtheorem*{rep@theorem}{\rep@title}
\newcommand{\newreptheorem}[2]{%
	\newenvironment{rep#1}[1]{%
		\def\rep@title{#2 \ref{##1}}%
		\begin{rep@theorem}}%
		{\end{rep@theorem}}}
\makeatother

\newreptheorem{theorem}{Theorem}
\newreptheorem{corollary}{Corollary}

\newtheorem{corollary}[theorem]{Corollary}
\theoremstyle{ddefinition}
\newtheorem{definition}[theorem]{Definition}

\numberwithin{equation}{section}

\definecolor{Blue}  {rgb} {0.282352,0.239215,0.803921}
\definecolor{Green} {rgb} {0.133333,0.545098,0.133333}
\definecolor{Red}   {rgb} {0.803921,0.000000,0.000000}
\definecolor{Violet}{rgb} {0.580392,0.000000,0.827450}

\usepackage[hidelinks]{hyperref}
\usepackage{mathtools}
\mathtoolsset{showonlyrefs}

\setkomafont{section}{\rmfamily\large}
\setkomafont{subsection}{\normalfont\slshape}
\setkomafont{subsubsection}{\rmfamily}
\setkomafont{subparagraph}{\normalfont\scshape}

\newtheorem*{theorem*}{Theorem}
\newtheorem*{corollary*}{Corollary}

\usepackage{titlesec}
\titleformat{\subsection}[runin]
{\normalfont\slshape}
{\thesubsection}
{0.5em}
{}
[.]

\makeatletter
\renewcommand\section{\@startsection {section}{1}{\z@}%
	{-3.5ex \@plus -1ex \@minus -.2ex}%
	{2.3ex \@plus.2ex}%
	{\normalfont\scshape\centering}}
\makeatother

\begin{document}
	
	\vspace*{-1.5cm}	\begin{center}	\textbf{\large{Frobenius Algebras, Factorization Homology \\[0.5ex] and the Reshetikhin-Turaev Invariants}}\\	\vspace{1cm}{\large Deniz Yeral }\\ 	\vspace{5mm}{\slshape  Université Bourgogne Europe\\ CNRS\\ IMB UMR 5584\\ F-21000 Dijon\\ France }\end{center}	\vspace{0.3cm}	
	\begin{abstract}\noindent 
		For a ribbon fusion category $\mathcal{A}$ and a special symmetric commutative Frobenius algebra $F$ in $\mathcal{A}$, we use factorization homology and the ansular correlators obtained via the modular microcosm principle to construct a diffeomorphism invariant vector inside the skein module of any closed oriented three-dimensional manifold. If $\mathcal{A}$ is a modular fusion category and $F$ is the monoidal unit, this recovers the Reshetikhin-Turaev invariants.
	\end{abstract}

	\section{Introduction and summary}
	The \emph{microcosm principle} in the sense of Baez-Dolan \cite{BaDo} roughly says that an algebraic structure of a certain type lives in a category equipped with a higher categorical version of the same algebraic structure. For example, an associative algebra can be defined in a monoidal category, which itself is an associative algebra up to coherent isomorphism
	in the bicategory of categories. 
	
	The microcosm principle is extended to cyclic and modular algebras in \cite{Woi}, and the following instance of it is going to be important in this note:
	Cyclic framed $E_2$-algebras $\mathcal{A}$ in a suitable symmetric monoidal bicategory of finite linear categories in the sense of~\cite{EO} over a fixed algebraically closed field $k$
	are classified in \cite{MW1}. In particular, a finite ribbon category, i.e.\ a finite braided monoidal category with rigid duality, simple unit and compatible balancing, is a cyclic framed $E_2$-algebra. The \emph{cyclic} microcosm principle allows us to consider cyclic framed $E_2$-algebras $F$ \emph{inside} a cyclic framed $E_2$-algebra $\mathcal{A}$, and this amounts exactly to a symmetric commutative Frobenius algebra. This is in line with the usual notion of Frobenius algebras in tensor categories, see e.g.\
	\cite{FuchsStigner}.
	
	A cyclic framed $E_2$-algebra $\mathcal{A}$ can be uniquely extended to an \emph{ansular functor}~\cite{MW2}, i.e.\ a consistent system of representations $\widehat{\mathcal{A}}(H)$ of the mapping class groups $\text{Map}(H)$ of compact oriented three-dimensional handlebodies $H$. Given a symmetric commutative Frobenius algebra $F$, the \emph{modular} microcosm principle endows each such handlebody group representation with a 
	vector $\xi_H^F \in \widehat{\mathcal{A}}(H)$ that is invariant with respect to the handlebody group action and compatible with gluing. In the case of a rigid duality, these handlebody group invariants extend the genus zero correlators in \cite{jfcs}. For this reason, they are referred to as \emph{ansular correlators} in \cite{Woi}. As opposed to the usual correlators in conformal field theory, see e.g.~\cite{algcften} for an introduction, the ansular correlators are not necessarily invariant under the full mapping class group of $\partial H$ (in case the action of $\text{Map}(H)$ on $\widehat{\mathcal{A}}(H)$ extends to $\text{Map}(\partial H)$), but just under $\text{Map}(H)\subset \text{Map}(\partial H)$. This reduced symmetry is crucial for the connection to manifold invariants. 
	
	In this note, we will treat the case in which $\mathcal{A}$ is a ribbon fusion category, in particular, $\mathcal{A}$ is semisimple. For the statement of the main result, suppose that $M\cong H'\cup_\Sigma H$ is a Heegaard splitting of a closed oriented three-dimensional manifold $M$ and $F\in \mathcal{A}$ is a symmetric commutative Frobenius algebra. Then we can consider the vector $\xi_H^F \in \widehat{\mathcal{A}}(H)$ and complement it  to a vector 
	$v_M^F \in \text{sk}_\mathcal{A}(M)$ inside the $\mathcal{A}$-skein module for $M$. If we additionally assume that $F$ is a \emph{special} Frobenius algebra in the sense of \cite{FFRS3}, we prove that this vector does not depend on the choice of the Heegaard splitting and is invariant with respect to diffeomorphisms of $M$:
	\begin{reptheorem}{thmain2}
		Let $M$ be a closed oriented three-dimensional manifold. Then any special symmetric commutative Frobenius algebra $F$ inside a ribbon fusion  category $\mathcal{A}$ gives rise to a vector $v_M^F \in \text{sk}_{\mathcal{A}}(M)$ inside the $\mathcal{A}$-skein module for $M$ that is invariant under oriented diffeomorphisms of $M$.
	\end{reptheorem}

	The pair of $\text{sk}_\mathcal{A}(M)$ with its pointing $v_M^F \in \text{sk}_\mathcal{A}(M)$ can be seen as an invariant of the manifold $M$. 
	The proof of Theorem~\ref{thmain2}
	relies on the factorization homology description of ansular functors and skein modules~\cite{Cooke,BW,BH,MW3}
	and demonstrates how manifold invariants can be obtained directly from the classification of operadic algebras, factorization homology and the microcosm principle. While some skein-theoretic calculations appear in the proof of the main result, they are not needed for the mere construction of the vector $v_M^F$.

	\subsection*{Relation to Reshetikhin-Turaev and Crane-Yetter topological field theories}
	Recall that a ribbon fusion category $\mathcal{A}$ gives rise to the four-dimensional Crane-Yetter topological field theory $Z_\mathcal{A}^\text{CY}$ \cite{CY,CKY,CGHP} that assigns to a closed oriented three-manifold $M$ the skein module $\text{sk}_\mathcal{A}(M)$. The pointing $v_M^F \in \text{sk}_\mathcal{A}(M)$
	can be understood as a boundary condition in the spirit of~\cite{Walker,FT,Ben} 
	(here just at the three-dimensional level)
	that 
	leads to a numerical manifold invariant if $\mathcal{A}$ is a \emph{modular} fusion category, i.e.\ has additionally
	a non-degenerate braiding. In that case, the value $\mathcal{Z}_\mathcal{A}^\text{CY}(W):\text{sk}_\mathcal{A}(M)\to k$ of the Crane-Yetter topological field theory on an oriented four-dimensional manifold $W$ with boundary $M$, evaluated at $v_M^F$, provides a numerical invariant of $M$ (the choice of $W$ leads to an ambiguity related to the framing anomaly). In the case where the Frobenius algebra $F$ is the monoidal unit of $\mathcal{A}$, the invariant $\mathcal{Z}_\mathcal{A}^\text{CY}(v_M^F)$ agrees with the Reshetikhin-Turaev invariant~\cite{RT1,RT2,Tur} of $M$ (Corollary~\ref{son}). The question of whether the invariant for $F\neq I$ is more powerful than the Reshetikhin-Turaev invariant as well as the generalizations to the non-semisimple, possibly even non-rigid case lie beyond the scope of this note.

	\subsection*{Inspiration from conformal field theory}
	It is, at least in the semisimple situation, standard to use topological tools to classify and construct correlation functions and to obtain partition functions as invariants of three-manifolds~\cite{FRS1,FRS2,FRS3,FRS4,FFRS,FFRS2}. 
	The construction in this note profits from a similar connection, but in the opposite direction: We take the ansular correlators, obtained from the modular microcosm principle, as a starting point to make assignments for three-dimensional manifolds.
	In doing so, it is crucial that the ansular correlators have a reduced symmetry.
	They are generally only invariant under the handlebody group, but not the entire mapping class group.

	\subsection*{Local modules and topological defects}
	Frobenius algebras with similar properties that we consider also arise in contexts different than the construction of correlators. 
	Namely, if a commutative Frobenius algebra in a ribbon or modular (fusion)
	category satisfies some symmetry and specialness conditions, then in many cases the category that consists of \emph{local modules} \cite{PAREIGIS1995413,SCHAUENBURG2001325} of the Frobenius algebra enjoys again nice properties, i.e.\  it is again ribbon or modular~\cite{kirillovostrik,FFRS3,LW,shimizuyadav}. This means that it could be fed again into the usual topological constructions. It is conceivable that the methods developed in this note admit a description in terms of categories of local modules. Although we do not include a precise statement here, we expect to get this relation through a comparison of our construction with the framework of \emph{defect topological field theories and their orbifolds}~\cite{crs1,CRS2,cmrss1}.
	In \cite{MR,cmrss2} the connection to 
	the once-extended topological field theory given by the category of local modules is established.
	Needless to say, any connection of the construction given in this article to defects and orbifolds would also be of independent interest.
	
	\subsection*{Acknowledgments} I would like to thank my advisor Lukas Woike for many helpful discussions and comments related to this project. Moreover, I would like to thank Christoph Schweigert for bringing the possible connection to local module categories to our attention, Vincentas Mulevičius for valuable explanations on defect topological field theories and their orbifolds, Edwin Kitaeff and Bertrand Patureau-Mirand for helpful discussions, and Lukas M\"uller for helpful comments on the manuscript. This project is supported by the ANR project CPJ n°ANR-22-CPJ1-0001-01 at the Institut de Mathématiques de Bourgogne. The IMB receives support from the EIPHI Graduate School (contract ANR-17-EURE-0002).

	\section{Preliminaries}
	In this preliminary section, we give a recollection of the main definitions and results that are involved in our construction. We start by introducing the formalism of cyclic and modular operads and algebras, together with a classification result of cyclic framed $E_2$-algebras and ansular functors. Then we recall the definition of factorization homology of surfaces and explain how ansular functors can be described in terms of factorization homology.
	\subsection{Cyclic and modular operads}
	Cyclic and modular operads were introduced by Getzler and Kapranov in \cite{GK1} and \cite{GK2}. Intuitively, a cyclic structure allows to consistently exchange the inputs and the output of operations. A modular structure moreover allows to self-compose operations. A definition of cyclic and modular operads is given in \cite{Cos} in terms of graph categories. This definition is adapted in \cite{MW1} to a bicategorical setting. In the following, we define cyclic and modular operads as symmetric monoidal functors from certain graph categories to the symmetric monoidal bicategory $\mathsf{Cat}$, the bicategory of categories, functors and natural transformations with the monoidal structure given by the Cartesian product.
	
	A finite graph $\Gamma$ is a finite set $H$ of half edges and a finite set $V$ of vertices together with a source map $s:H \rightarrow V$ which assigns to a half edge the vertex that it is attached to and an involution $\iota : H\rightarrow H$ which encodes the way half edges are glued. The orbits of $\iota$ are called \emph{edges}. The orbits with two elements are called \emph{internal edges}. Half edges which are fixed by $\iota$ are called \emph{external legs}. The set of external legs is denoted by $\text{Legs}(\Gamma)$. A finite graph with one vertex and no internal edges is called a \emph{corolla}. A map between two graphs $(H,V,\iota, s)$ and $(H',V',\iota', s')$ consists of two maps $\varphi:H \rightarrow H'$ and $\psi:V \rightarrow V'$ so that $\varphi \iota = \iota' \varphi$ and $\psi s = s' \varphi$. 
	
	Let us define $\mathsf{Graphs}$, the category of graphs: The objects are finite disjoint unions of corollas.  A morphism $T \rightarrow T'$ is an equivalence class of graphs $\Gamma$ together with maps $\nu(\Gamma) \cong T$ and $\pi_0(\Gamma) \cong T'$, where $\nu(\Gamma)$ is the graph obtained by cutting in half all the internal edges and $\pi_0(\Gamma)$ is the graph obtained by contracting all the internal edges. The composition of two morphisms $\Gamma_2 \circ \Gamma_1$ is the graph given by replacing the vertices of $\Gamma_2$ with $\Gamma_1$. The disjoint union endows $\mathsf{Graphs}$ with a symmetric monoidal structure.

	A (category-valued) \emph{modular operad} $\mathcal{O}$ is a symmetric monoidal functor $\mathcal{O}: \mathsf{Graphs} \rightarrow \mathsf{Cat}$. Note that here $\mathsf{Graphs}$ is seen as a symmetric monoidal \emph{bicategory} and $\mathcal{O}$ is a symmetric monoidal functor between \emph{bicategories}, meaning that it should be understood in a weak sense, i.e.\ up to coherent isomorphism, see \cite[Section 2]{SP} for the precise coherence data. 
	Note moreover that this definition so far does not include operadic identities; for how to include them,  we refer to~\cite[Definition 2.3]{MW1}. The operads that we consider in this note will always have operadic identities.

	For a corolla $T$, we call the objects in $\mathcal{O}(T)$ the \emph{operations of total arity $|\text{Legs}(T)|$}, where $|\text{Legs}(T)|$ denotes the cardinality of the set $\text{Legs}(T)$ of $T$. The morphisms of $\mathsf{Graphs}$ that are edge contractions give operadic compositions. In order to define cyclic operads, we can similarly construct the category $\mathsf{Forests}$ which has the same objects as $\mathsf{Graphs}$, but only with morphisms whose connected components are contractible. Then by replacing $\mathsf{Graphs}$ above with $\mathsf{Forests}$ we get the notion of a \emph{cyclic operad}.
	
	\subsection{Cyclic and modular algebras}
	
	An algebra over an operad is defined by introducing the so-called \emph{endomorphism operad}. The endomorphism operad has an underlying object $X$ in an ambient symmetric monoidal bicategory. The operations in total arity $(n+1)$ can be seen as morphisms $X^{\otimes n} \to X$ and the operadic composition is given by composition of morphisms. Cyclic and modular endomorphism operads are defined similarly, but with a self-duality structure on the underlying object that captures the cyclic structure and self-compositions. This is introduced in \cite{GK1} and \cite{GK2}. The definitions are extended to the bicategorical situation in \cite{MW1}. 
	
	The ambient symmetric monoidal bicategory that we consider is $\mathsf{Rex}^{\text{f}}$: The objects are \emph{finite} categories in the sense of \cite{EO}, i.e.\ $k$-linear abelian categories whose objects have finite length and which have finite dimensional morphism spaces, enough projective objects and finitely many isomorphism classes of simple objects. A category is finite if and only if it is equivalent to the category of finite dimensional modules of a finite dimensional algebra. The 1-morphisms of $\mathsf{Rex}^{\text{f}}$ are right exact functors. The 2-morphisms are linear natural transformations. The monoidal product is the Deligne tensor product $\boxtimes$, see e.g.\ \cite{ILM} for more background. The monoidal unit is the category of finite-dimensional vector spaces that we denote by $\mathsf{Vect}$. 
	
	Consider an object $X \in \mathsf{Rex}^{\text{f}}$ together with a \emph{non-degenerate symmetric} pairing $\kappa: X \boxtimes X \rightarrow \mathsf{Vect}$. Here, non-degenerate means that there exists a copairing $\Delta: \mathsf{Vect} \rightarrow X \boxtimes X$ so that usual snake relations hold up to isomorphism. The pairing is symmetric in the sense that it is equipped with a $\mathbb{Z}_2$-homotopy fixed point structure with respect to the action of the symmetric braiding of $\mathsf{Rex}^{\text{f}}$. 
	
	We define the modular endomorphism operad $\mathsf{End}^{X}_{\kappa}$ as follows: Let $T$ be a corolla. We set $\mathsf{End}^{X}_{\kappa}(T):=\text{Hom}_{\mathsf{Rex}^{\text{f}}}(X^{\boxtimes \text{Legs}(T)},\mathsf{Vect})$. The value of $\mathsf{End}^{X}_{\kappa}$ on a graph that is an edge contraction is defined by inserting $\Delta$ on the internal edges of the graph, see \cite[Section 2.3]{MW1} for the precise definition. A \emph{modular algebra over a modular operad} $\mathcal{O}$ is an object $X \in \mathsf{Rex}^{\text{f}}$ with a non-degenerate symmetric pairing $\kappa$ together with a symmetric monoidal transformation $\mathcal{A}:\mathcal{O} \rightarrow \mathsf{End}_{\kappa}^{X}$. In particular, for any operation $o \in \mathcal{O}(T)$ we get a right exact functor $\mathcal{A}_{o}:X^{\boxtimes \text{Legs}(T)} \rightarrow \mathsf{Vect}$ together with an $\text{End}_{\mathcal{O}(T)}(o)$-action. Note that if $o$ is an operation in total arity 0, then $\mathcal{A}_{o}$ can be identified with a vector space. The notion of a \emph{cyclic} algebra is defined in the same way. Note that in general we will denote the algebra and its underlying category with the same notation.
	
	The definition of an algebra over an operad with an operadic identity has extra compatibility condition. In particular, one has to specify an isomorphism $\mathcal{A}_{1_{\mathcal{O}}} \cong \kappa$ where $1_{\mathcal{O}}$ denotes the operadic identity of $\mathcal{O}$. 
	
	Modular algebras over a modular operad $\mathcal{O}$ form a 2-groupoid \cite[Proposition 2.18]{MW1} and we call two modular algebras equivalent if they are equivalent in this 2-groupoid. The notion of equivalence between cyclic algebras is defined likewise.

	\subsection{Classification of cyclic framed $E_2$-algebras and ansular functors}
	The following two operads will be relevant in this note:
	\begin{itemize}
		\item The modular handlebody operad $\mathsf{Hbdy}$ \cite{Gia}: Let $T$ be a corolla. The objects of the category $\mathsf{Hbdy}(T)$ are connected oriented handlebodies $H$ with $\text{Legs}(T)$-parametrized disks on the boundary, i.e.\ an orientation preserving embedding $\mathbb{D}^{\sqcup \text{Legs}(T)} \hookrightarrow \partial H$ where $\mathbb{D}$ denotes the two-dimensional disk with a fixed orientation and $\sqcup \text{Legs}(T)$ denotes the unordered disjoint union over the set $\text{Legs}(T)$. The morphisms in $\mathsf{Hbdy}(T)$
		are isotopy classes of orientation preserving diffeomorphisms of handlebodies that also preserves the boundary parametrizations. We can monoidally extend this definition to disjoint unions of corollas. The operadic composition is given by gluing along the boundary disks. Note that the automorphism group $\text{Aut}_{\mathsf{Hbdy}(T)}(H)$ is the handlebody group of $H$, see \cite{henselprimer} for an overview of handlebody groups.
		\item The framed $E_2$-operad $\mathsf{fE_2}$ \cite{BV1,May,BV2}: We define the category of operations in total arity $(n+1)$ to be the fundamental groupoid of the space of embeddings $\mathbb{D}^{\sqcup n} \hookrightarrow \mathbb{D}$ that are composed of translations, rescalings and rotations. The operadic composition is induced by composition of embeddings. Note that the framed $E_2$-operad is usually defined as a topological operad, that is, a symmetric monoidal functor $\mathsf{Graphs} \rightarrow \mathsf{Top}$ in the usual 1-categorical sense, the latter is the category of topological spaces. The space of operations in total arity $(n+1)$ is defined as space of embeddings described above. The reason why we can equivalently work with the category-valued model is that the topological framed $E_2$-operad is aspherical \cite{Wahl,SalWahl}, i.e.\ homotopy groups of the spaces of embeddings $\mathbb{D}^{\sqcup n} \hookrightarrow \mathbb{D}$ that are higher than degree 1 vanish.
	\end{itemize}
	
	By restricting the handlebody operad to genus zero, we obtain the \emph{cyclic} framed $E_2$-operad \cite{Bud,Wahl,SalWahl}.
	This implies that modular $\mathsf{Hbdy}$-algebras, the so-called \emph{ansular functors},
	produce by genus zero restriction a cyclic framed $E_2$-algebra.
	This restriction is an equivalence as shown in \cite[Theorem 5.3]{MW2} using \cite{Cos,Gia,MW1}.
	We denote the unique ansular functor extending a cyclic framed $E_2$-algebra $\mathcal{A}$ by $\widehat{\mathcal{A}}$.
	
	Cyclic framed $E_2$-algebras in $\mathsf{Rex}^{\text{f}}$ are classified by \emph{ribbon Grothendieck-Verdier categories} $\mathcal{A}$ in $\mathsf{Rex}^{\text{f}}$ \cite[Theorem 5.13]{MW1}. The notion of a Grothendieck-Verdier category was introduced in \cite{BoDr} and studied before under the name of $*$-autonomous categories \cite{Barr}. Very roughly, a Grothendieck-Verdier category is a monoidal category with a weak duality notion --- in particular, a rigid duality is a Grothendieck-Verdier duality. In our context, this weak duality is induced by the non-degenerate pairing $\kappa$ of $\mathcal{A}$ that is included in the definition of a cyclic algebra --- we refer to \cite[Section 2.5]{MW1}, together with \cite{FSS}, to see how this duality arises. Note that \cite{MW1} investigates the non-degenerate pairings in the bicategory $\mathsf{Lex}^{\text{f}}$, which consists of finite linear categories, \emph{left} exact functors and natural transformations. This leads to dual statements in our $\mathsf{Rex}^{\text{f}}$-valued case, see \cite{BW} and \cite{MW3} where our convention is also used.

	A finite ribbon category $\mathcal{A}$ is a particular example of a ribbon Grothendieck-Verdier category in $\mathsf{Rex}^{\text{f}}$, and it will be the only class of examples that we will consider in this note. In particular, $\mathcal{A}$ has a monoidal product $\otimes: \mathcal{A} \boxtimes \mathcal{A} \to \mathcal{A}$ with a simple unit $I$ together with
	\begin{itemize}
		\item a rigid duality, i.e.\ every object has left and right duals in the sense of \cite[Section 2.10]{EGNO},
		\item natural isomorphisms $c_{X,Y}:X\otimes Y \cong Y \otimes X$ and $\theta_{X}:X \cong X$, called the braiding and the balancing, respectively, such that $c$ satisfies the usual hexagon relations, $\theta_{X \otimes Y} =  c_{Y,X} c_{X,Y} (\theta_{X} \otimes \theta_{Y})$ and $ \theta_{I} = \text{id}_{I}$,
		\item $\theta_X^\vee = \theta_{X^\vee}$ where $X^\vee$ denotes the left dual of $X$.
	\end{itemize}
	
	\subsection{Factorization homology and ansular functors}
	\emph{Factorization homology} is a homology theory for topological manifolds with coefficients in $E_n$-categories \cite{Lur,AF} inspired by \cite{BeDr}. Given an $E_n$-algebra $\mathcal{A}$ in a suitable higher category, factorization homology assigns to each $n$-manifold $\Sigma$ an object $\int_{\Sigma}\mathcal{A}$. This assignment defines a homology theory in the sense that it satisfies an analogue of Eilenberg-Steenrod axioms, in particular, it has nice gluing properties.
	
	In this note, the relevant case is factorization homology for oriented two-dimensional manifolds~\cite{BZBJ}. Factorization homology of surfaces takes a framed $E_2$-algebra $\mathcal{A}$ as input and assigns to a surface $\Sigma$ (for us, a surface will be always compact and oriented) an object $\int_{\Sigma}\mathcal{A} \in \mathsf{Rex}^{\text{f}}$ in a functorial way with respect to oriented embeddings. It is defined as
	\begin{equation*}
		\int_{\Sigma}\mathcal{A} := \underset{\substack{\varphi:\mathbb{D}^{\sqcup n}\hookrightarrow \Sigma \\ n \geq 0}}{\operatorname{hocolim}}\mathcal{A}^{\boxtimes n}
	\end{equation*}
	where the homotopy colimit runs over all oriented embeddings of disjoint unions of disks into $\Sigma$. The functor $\mathcal{O}^{\mathcal{A}}_{\Sigma}:\mathsf{Vect} \rightarrow \int_{\Sigma}\mathcal{A}$ induced by the embedding $\emptyset \hookrightarrow \Sigma$ can be identified with an object in $\int_{\Sigma}\mathcal{A}$ that is again denoted by $\mathcal{O}^{\mathcal{A}}_{\Sigma}$ and called the \emph{quantum structure sheaf}.
	
	If $\mathcal{A}$ is moreover a \emph{cyclic} framed $E_2$-algebra, then we obtain even more structure thanks to the ansular functor $\widehat{\mathcal{A}}$, see~\cite[Section~4]{BW} for the details: Let $H$ be a handlebody without boundary disks and $\Sigma = \partial H$. One can construct a	functor
	\begin{equation*}
		\Phi_{\mathcal{A}}(H):\int_{\Sigma}\mathcal{A} \rightarrow \mathsf{Vect}
	\end{equation*}
	by using 
	$\widehat{\mathcal{A}}$ and the universal property of factorization homology. Then, the following diagram commutes up to a canonical isomorphism:
	\begin{align}\begin{array}{c}
			\begin{tikzpicture}[scale=0.5]
				\begin{pgfonlayer}{nodelayer}
					\node [style=none] (0) at (-9.5, 0) {};
					\node [style=none] (1) at (0, 0) {};
					\node [style=none] (2) at (10, 0) {};
					\node [style=none] (3) at (-9.5, 0) {$\mathsf{Vect}$};
					\node [style=none] (4) at (0, 0) {};
					\node [style=none] (5) at (0, 0) {$\int_{\Sigma}\mathcal{A}$};
					\node [style=none] (6) at (9, 0) {$\mathsf{Vect}$};
					\node [style=none] (7) at (-8.25, 0) {};
					\node [style=none] (8) at (-1.75, 0) {};
					\node [style=none] (9) at (1.5, 0) {};
					\node [style=none] (10) at (8, 0) {};
					\node [style=none] (11) at (-8.75, 0.75) {};
					\node [style=none] (12) at (8, 0.75) {};
					\node [style=none] (14) at (0, 2.5) {$\hat{\mathcal{A}}(H)$};
					\node [style=none] (16) at (0, 1) {$\cong$};
					\node [style=none] (17) at (-5, -0.5) {$\mathcal{O}^{\mathcal{A}}_{\Sigma}$};
					\node [style=none] (18) at (5, -0.5) {$\Phi_{\mathcal{A}}(H)$};
				\end{pgfonlayer}
				\begin{pgfonlayer}{edgelayer}
					\draw [style=endarrow] (7.center) to (8.center);
					\draw [style=endarrow] (9.center) to (10.center);
					\draw [style=endarrow, bend left, looseness=0.50] (11.center) to (12.center);
				\end{pgfonlayer}
			\end{tikzpicture}			
		\end{array}	
		\label{factandans}	\end{align} 
	This gives the factorization homology description of ansular functors.
	
	Let $\mathcal{A}$ be a ribbon \emph{fusion} category, i.e.\ a \emph{semisimple} finite ribbon category, and $M\cong H' \cup_\Sigma H$ a Heegaard splitting for a closed oriented three-dimensional manifold $M$. By \cite[Corollary 3.10]{MW3}, there is a canonical isomorphism 
	\begin{align}
		\int^{X \in \int_\Sigma \mathcal{A}} \Phi_\mathcal{A}(H')X^\vee \otimes \Phi_\mathcal{A}(H)X \xrightarrow{\ \cong\ } \text{sk}_\mathcal{A}(M) \ , \label{eqnphiskeiniso}
	\end{align}
	where $\text{sk}_\mathcal{A}(M)$ is the skein module of the manifold $M$ for the ribbon category $\mathcal{A}$. Recall that the skein module $\text{sk}_\mathcal{A}(M)$ is, roughly, the vector space spanned by all ribbons with coupons in $M$ whose strands are labeled by objects in $\mathcal{A}$ while the coupons are labeled by the morphisms in $\mathcal{A}$, considered up to skein relations, see e.g.\ \cite{GJS} and the references therein.
	
	\section{construction of the vectors inside the skein modules of three-dimensional manifolds}
	
	The \emph{microcosm principle} \cite{BaDo} allows one to consider algebras \emph{inside} a higher categorical version of the same algebraic structure, here an algebra should be understood in the sense of an algebra over an operad. For example, an algebra in the bicategory of categories over the associative operad is a category $\mathcal{A}$, \emph{the macrocosm}, together with a product $\otimes: \mathcal{A} \times \mathcal{A} \to \mathcal{A}$ that is associative up to coherent homotopy, i.e.\ $\mathcal{A}$ is a monoidal category. One can define an associative algebra inside the associative algebra $\mathcal{A}$ to be an object $A \in \mathcal{A}$, \emph{the microcosm}, together with a product $A \otimes A \to A$, that is associative with respect to the associativity isomorpshisms of $\mathcal{A}$. This is the usual notion of an associative algebra $A$ inside a monoidal category $\mathcal{A}$. 
	
	The microcosm principle is extended to cyclic and modular algebras in \cite{Woi}, here we will remind the relevant definitions and results for cyclic framed $E_2$-algebras and ansular functors, we refer to \cite[Section 5 \& Section 10]{Woi} for the precise statements: Recall that cyclic and modular algebras are defined through a self-duality on the underlying object, i.e.\ via a non-degenerate pairing $\kappa$ on $\mathcal{A}$. Analogously, in order to capture the cyclic structure at the microcosm level, one needs to ask for \emph{self duality} of the microcosm $F$, which consists of a map $\beta: F \boxtimes F \to \Delta$ that is subject to non-degeneracy and symmetry conditions. A cyclic framed $E_2$-algebra inside a cyclic framed $E_2$-algebra $\mathcal{A}$ is a self-dual object $F \in \mathcal{A}$ together with a collection of vectors $\xi_{\Sigma}^F \in \mathcal{A}(\Sigma)(F,\dots,F)$ for every genus zero surface $\Sigma$. These vectors are invariant under the mapping class group action and compatible with gluing, which is defined via $\beta$. The \emph{modular} microcosm principle moreover extends a cyclic framed $E_2$-algebra $F$ in $\mathcal{A}$ to handlebodies, i.e.\ to a collection of handlebody group invariant vectors $\xi_{H}^F \in \widehat{\mathcal{A}}(H)(F,\dots ,F)$ that respect gluing, we refer to them as \emph{ansular correlators}. This follows from a microcosmic version of the equivalence between cyclic framed $E_2$-algebras and ansular functors. Cyclic framed $E_2$-algebras $F$ inside $\mathcal{A}$ are symmetric commutative Frobenius algebras --- for example, the vector $\xi_{\mathbb{S}^1 \times [0,1]}^F \in \mathcal{A}(\mathbb{S}^1 \times [0,1])(F,F)$ gives the non-degenerate invariant pairing of the Frobenius algebra $F$. The fact that this vector is invariant with respect to the mapping class group action implies that the pairing is symmetric. We refer to e.g.\ \cite{FuchsStigner} for the definition of a symmetric commutative Frobenius algebra in the rigid case, an example of a symmetric commutative Frobenius algebra that we will particularly consider is the monoidal unit $I$. 
	
	Let $\mathcal{A}$ be a ribbon fusion category. In particular, $\mathcal{A}$ is a cyclic framed $E_2$-algebra and uniquely extends to an ansular functor $\widehat{\mathcal{A}}$ in $\mathsf{Rex}^{\text{f}}$.
	Let $M$ be a closed oriented three-manifold with a Heegaard splitting $M \cong H' \cup_{\Sigma} H$. The factorization homology description of ansular functors in~\eqref{factandans} gives the map	
	\small
	\begin{align}
		\iota_{H',H}: \widehat{\mathcal{A}}(H') \otimes \widehat{\mathcal{A}}(H) \stackrel{\eqref{factandans}}{\cong} \Phi_\mathcal{A}(H')\mathcal{O}_{\bar{\Sigma}}^\mathcal{A} \otimes \Phi_\mathcal{A}(H)\mathcal{O}^{\mathcal{A}}_{\Sigma} \xrightarrow{(*)} \int^{X \in \int_\Sigma \mathcal{A}} \Phi_\mathcal{A}(H')X^\vee \otimes \Phi_\mathcal{A}(H)X \stackrel{\eqref{eqnphiskeiniso}}{\cong} \text{sk}_\mathcal{A}(M)
	\end{align}
	\normalsize
	where $(*)$ is the structure map of the coend. Using the ansular correlators, we can make the following definition:
	\begin{definition}\label{defvmf}
		Let $F\in\mathcal{A}$ be a  symmetric commutative Frobenius algebra in a ribbon fusion category.
		For a closed oriented three-dimensional manifold $M$, choose  a Heegaard splitting $M \cong H' \cup_\Sigma H$. 
		Through
		\begin{equation*}
			k \xrightarrow{\ \xi_{H'}^I \otimes \xi_H^F\ } \widehat{\mathcal{A}}(H') \otimes \widehat{\mathcal{A}}(H) \xrightarrow{ \iota_{H',H} } \text{sk}_\mathcal{A}(M) \ ,
		\end{equation*}
		we define a vector
		\begin{align}v_M^F \in \text{sk}_\mathcal{A}(M)
		\end{align}
		in the skein module of $M$.
	\end{definition}
	Let us assume furthermore that $F$ is a \emph{special} Frobenius algebra \cite[Definition 2.22]{FFRS3}, i.e.\ for non-zero scalars $\lambda, \lambda'$; one has $\varepsilon\eta = \lambda \text{id}_{I}$ and $\mu\delta = \lambda' \text{id}_{F}$, with $\eta: I \to F$ the unit, $\varepsilon:F \to I$ the counit, $\mu:F\otimes F \to F$ the product and $\delta:F \to F \otimes F$ the coproduct of $F$. We can renormalize $\varepsilon$ and $\delta$ to have $\mu \delta = \text{id}_{F}$. With this renormalization, one can show that the vector $v_M^F$ does not depend on the Heegaard splitting of $M$ and we have the following theorem:
	\begin{theorem}\label{thmain2}
		Let $M$ be a closed oriented three-dimensional manifold.
		Then any special symmetric commutative Frobenius algebra $F$ inside a ribbon fusion category $\mathcal{A}$ gives rise to a vector $v_M^F \in \text{sk}_\mathcal{A}(M)$ inside the $\mathcal{A}$-skein module for $M$ that is additionally invariant under oriented diffeomorphisms of $M$.
	\end{theorem}

	\begin{proof}
		For the proof, we will use the Reidemeister-Singer Theorem, see e.g.~\cite{Craggs}. More precisely, we will show that the vector from Definition~\ref{defvmf} \begin{enumerate}[label={(\roman{*})}]
			\item descends to the mapping classes,		\label{hs3}
			\item descends to double cosets under the handlebody group,     \label{hs1}
			\item and remains stable after adding genus one handlebodies to the Heegaard splitting that are each other's complement in $\mathbb{S}^3$. \label{hs2}
		\end{enumerate}
		The proof will proceed by recalling / explaining what is meant by~\ref{hs3},~\ref{hs1} and~\ref{hs2}. Then we will show that $v_M^F$ is compatible with these operations.
		
		For~\ref{hs3} and~\ref{hs1}, we will use the following reformulation of Heegaard splittings, see e.g.~\cite[Section~5.8]{BK} and~\cite[Section~4]{henselprimer}: We can assume that $H'=f.{\bar H}$, where $\bar H$ is obtained from $H$ by orientation reversal, $f \in \text{Diff}(\bar \Sigma)$ and $f.\bar H$ is defined through the pushout
		\begin{equation}	
			\begin{tikzcd}
				\bar	\Sigma\times\{0\} \cong \bar \Sigma=\partial \bar H \ar[rr] \ar[dd,"f\times 0",swap] && \bar H \ar[dd, "\tilde{f}"] \\ \\ 
				\bar	\Sigma\times [0,1] \ar[rr, swap]  & & f.\bar H \ ,
			\end{tikzcd} \label{ftilde}
		\end{equation}
		see also \cite[Section~2.3]{BW}.
		
		The condition~\ref{hs3} means that, given two diffeomorphisms $f,f' \in \text{Diff}(\bar{\Sigma})$ that are isotopic, then \begin{align}  f.\bar{H} \cup_\Sigma H \cong f'.\bar{H} \cup_\Sigma H \ . \end{align}
		Hence, we need to show that these two Heegaard splittings give rise to the same element $v^F_M$ inside the skein module of $M$. Recall that, by \cite[Proposition 4.6]{BW}, one has a canonical isomorphism \begin{align} \Phi_{\mathcal{A}}(f.\bar{H})f_* \cong \Phi_{\mathcal{A}}(\bar{H}) \end{align}
		where $f_*:\int_{\bar{\Sigma}}\mathcal{A} \to \int_{\bar{\Sigma}}\mathcal{A}$ denotes the induced map on factorization homology. The evaluation of this isomorphism at $\mathcal{O}_{\bar{\Sigma}}$ is the map $\widehat{\mathcal{A}}(\tilde{f}):\widehat{\mathcal{A}}(\bar{H}) \cong \widehat{\mathcal{A}}(f.\bar{H})$ where $\tilde{f}$ is given by~\eqref{ftilde}, this follows from the homotopy fixed point structure of $\mathcal{O}_{\bar{\Sigma}}$. 
		An isotopy $\gamma:f \cong f'$ gives an isomorphism $\gamma_*:f_* \cong f'_*$ which leads to:
		\begin{align}
			\Phi_{\mathcal{A}}(f.\bar{H}) \cong \Phi_{\mathcal{A}}(\bar{H})f^{-1}_* \stackrel{\gamma_*}{\cong} \Phi_{\mathcal{A}}(\bar{H})f'^{-1}_* \cong \Phi_{\mathcal{A}}(f'.\bar{H}) \label{eqnisotopy}
		\end{align}
		This induces a map \begin{align}
			\widehat{\mathcal{A}}(f.\bar{H}) \cong \Phi_{\mathcal{A}}(f.\bar{H})\mathcal{O}_{\bar{\Sigma}}^{\mathcal{A}} \cong \Phi_{\mathcal{A}}(f'.\bar{H})\mathcal{O}_{\bar{\Sigma}}^{\mathcal{A}} \cong \widehat{\mathcal{A}}(f'.\bar{H}). \end{align}
		This map coincides with the composition $\widehat{\mathcal{A}}(\tilde{f'}) \circ \widehat{\mathcal{A}}(\tilde{f})^{-1}$. In particular, it maps $\xi_{f.\bar{H}}$ to $\xi_{f'.\bar{H}}$, as the ansular correlators are preserved by the maps between handlebodies. Then, the isomorphism $\Phi_{\mathcal{A}}(f.\bar{H})\cong \Phi_{\mathcal{A}}(f'.\bar{H})$ that is given by~\eqref{eqnisotopy}, which induces a map between coends, preserves the vectors that we chose, hence the condition~\ref{hs3} is verified.
		
		The condition~\ref{hs1} is equivalent to the following: For any $a, b \in \text{Map}(\bar{H}) \subset \text{Map}(\bar{\Sigma})$\begin{align}  afb.\bar{H} \cup_\Sigma H \cong f.\bar{H} \cup_\Sigma H \ . \end{align} Again, we will show that these Heegaard splittings point to the same element $v^F_M$ inside the skein module of $M$. First, observe that we have 
		\begin{equation*}
			\int^{X \in \int_{\Sigma}\mathcal{A}}\Phi_{\mathcal{A}}(afb.\bar{H})X^\vee \otimes \Phi_{\mathcal{A}}(H)X \cong \int^{X \in \int_{\Sigma}\mathcal{A}} \Phi_{\mathcal{A}}(fb.\bar{H})X^\vee \otimes \Phi_{\mathcal{A}}(a^{-1}.H)X
		\end{equation*}
		This is induced by the equivalence $a_*:\int_{\Sigma}\mathcal{A} \simeq \int_{\Sigma}\mathcal{A}$ and the canonical isomorphism $\Phi_{\mathcal{A}}(H) a_* \cong \Phi_{\mathcal{A}}(a^{-1}.H)$. This induces a map from $\widehat{\mathcal{A}}(afb.\bar{H})\otimes \widehat{\mathcal{A}}(H)$ to $\widehat{\mathcal{A}}(fb.\bar{H})\otimes \widehat{\mathcal{A}}(a^{-1}.H)$ that commutes with the inclusions into the coends. Then, we can conclude with the observation that, since $a$ and $b$ are elements of the handlebody group, we have canonical diffeomorphisms $a^{-1}.H \cong H$ and $b.\bar{H} \cong \bar{H}$ \cite[Remark 2.8]{BW}, and the isomorphism between $\widehat{\mathcal{A}}(afb.\bar{H})\otimes \widehat{\mathcal{A}}(H)$ and $\widehat{\mathcal{A}}(fb.\bar{H})\otimes \widehat{\mathcal{A}}(a^{-1}.H)$ is given by $\widehat{\mathcal{A}}(a^{-1}) \otimes \widehat{\mathcal{A}}(a)$ \cite[Corollary 4.7]{BW}. This map preserves the vectors given by the microcosm principle and this concludes the proof of~\ref{hs1}.

		For~\ref{hs2}, we will give a geometric proof that is based on the comparison between skein modules and ansular functors. More precisely, we will use the fact that $\widehat{\mathcal{A}}(H)$ can be seen as the skein module of $H$ by \cite[Theorem 3.4]{MW3}. Let us observe that the vector 
		$\xi_H^F\in\widehat{\mathcal{A}}(H)$
		inside the skein module for the handlebody $H$
		is given by a skein inside $H$ colored by $F$
		that is obtained by the following combination of the unit $\eta : I \to F$, the product $\mu : F \otimes F \to F$, the counit $\varepsilon : F \to I$ and the coproduct $\delta : F \to F \otimes F$, to be read from left to right:
		\begin{align}\begin{array}{c}	
				\begin{tikzpicture}[scale=0.5]
					\begin{pgfonlayer}{nodelayer}
						\node [style=none] (0) at (8, 1.75) {};
						\node [style=none] (1) at (8, -2.25) {};
						\node [style=none] (2) at (-9, 1.75) {};
						\node [style=none] (3) at (-9, -2.25) {};
						\node [style=none] (4) at (-7.5, -0.5) {};
						\node [style=none] (5) at (-6.5, -0.5) {};
						\node [style=none] (6) at (-8, 0) {};
						\node [style=none] (7) at (-6, 0) {};
						\node [style=none] (8) at (-2.5, -0.5) {};
						\node [style=none] (9) at (-1.5, -0.5) {};
						\node [style=none] (10) at (-3, 0) {};
						\node [style=none] (11) at (-1, 0) {};
						\node [style=none] (12) at (4.5, -0.5) {};
						\node [style=none] (13) at (5.5, -0.5) {};
						\node [style=none] (14) at (4, 0) {};
						\node [style=none] (15) at (6, 0) {};
						\node [style=smallblue] (16) at (-9, -0.25) {\color{white}$\delta$};
						\node [style=smallblue] (17) at (-5, -0.25) {\color{white}$\mu$};
						\node [style=smallblue] (18) at (-4, -0.25) {\color{white}$\delta$};
						\node [style=smallblue] (19) at (0, -0.25) {\color{white}$\mu$};
						\node [style=smallblue] (20) at (3, -0.25) {\color{white}$\delta$};
						\node [style=smallblue] (21) at (7, -0.25) {\color{white}$\mu$};
						\node [style=smallblue] (22) at (-10, -0.25) {\color{white}$\eta$};
						\node [style=smallblue] (23) at (8, -0.25) {\color{white}$\varepsilon$};
						\node [style=none] (24) at (1.5, -0.25) {\dots};
						\node [style=none] (25) at (1, -0.25) {};
						\node [style=none] (26) at (2, -0.25) {};
						\node [style=none] (27) at (-7, -1.75) {$H$};
						\node [style=none] (29) at (-8.5, 1) {$F$};
					\end{pgfonlayer}
					\begin{pgfonlayer}{edgelayer}
						\draw [bend left=45, looseness=1.25] (4.center) to (5.center);
						\draw [bend right=60, looseness=1.50] (6.center) to (7.center);
						\draw [bend left=45, looseness=1.25] (8.center) to (9.center);
						\draw [bend right=60, looseness=1.50] (10.center) to (11.center);
						\draw [bend left=45, looseness=1.25] (12.center) to (13.center);
						\draw [bend right=60, looseness=1.50] (14.center) to (15.center);
						\draw [bend right=90, looseness=1.75] (2.center) to (3.center);
						\draw [bend left=90, looseness=1.75] (0.center) to (1.center);
						\draw (2.center) to (0.center);
						\draw (3.center) to (1.center);
						\draw [style=blue, bend left=90, looseness=0.75] (16) to (17);
						\draw [style=blue] (17) to (18);
						\draw [style=blue, bend left=90, looseness=0.75] (18) to (19);
						\draw [style=blue, bend left=90, looseness=0.75] (20) to (21);
						\draw [style=blue, bend right=90, looseness=0.75] (16) to (17);
						\draw [style=blue, bend right=90, looseness=0.75] (18) to (19);
						\draw [style=blue, bend right=90, looseness=0.75] (20) to (21);
						\draw [style=blue] (16) to (22);
						\draw [style=blue] (21) to (23);
						\draw [style=blue] (19) to (25.center);
						\draw [style=blue] (26.center) to (20);
					\end{pgfonlayer}
				\end{tikzpicture}
				\label{eqnstabilization_elaborate}\end{array}
		\end{align} 
		This is a consequence of the construction in \cite[Definition 5.1]{Woi} together with the fact that excision for the skein modules is induced by gluing along boundary disks.
		Note that the translation of $\xi_H^F$ into a combination of $\eta, \mu, \delta$ and $\varepsilon$ depends on the reading direction, but the overall element $\xi_H^F$ does not --- this is a consequence of $F$ being a cyclic framed $E_2$-algebra inside a cyclic framed $E_2$-algebra.
		In other words, even though the picture in~\eqref{eqnstabilization_elaborate}
		suggest the choice of a system of cuts decomposing $H$ into genus zero handlebodies, $\xi_F^H$ does not depend on this choice.
		
		In the sequel, we will omit the labels for $\mu$ and $\delta$ (we will just a draw a trivalent vertex) and for $\eta$ and $\varepsilon$ as well (we will just draw a small dot).
		With these conventions, the skein in $M$ describing $v^F_M$ is given by:
		\begin{align} \begin{tikzpicture}[scale=0.5]
				\begin{pgfonlayer}{nodelayer}
					\node [style=none] (0) at (7.75, 2) {};
					\node [style=none] (1) at (7.75, -2) {};
					\node [style=none] (2) at (-9.25, 2) {};
					\node [style=none] (3) at (-9.25, -2) {};
					\node [style=none] (4) at (-7.75, -0.25) {};
					\node [style=none] (5) at (-6.75, -0.25) {};
					\node [style=none] (6) at (-8.25, 0.25) {};
					\node [style=none] (7) at (-6.25, 0.25) {};
					\node [style=none] (8) at (-2.75, -0.25) {};
					\node [style=none] (9) at (-1.75, -0.25) {};
					\node [style=none] (10) at (-3.25, 0.25) {};
					\node [style=none] (11) at (-1.25, 0.25) {};
					\node [style=none] (12) at (4.25, -0.25) {};
					\node [style=none] (13) at (5.25, -0.25) {};
					\node [style=none] (14) at (3.75, 0.25) {};
					\node [style=none] (15) at (5.75, 0.25) {};
					\node [style=none] (16) at (-9.25, 0) {};
					\node [style=none] (17) at (-5.25, 0) {};
					\node [style=none] (18) at (-4.25, 0) {};
					\node [style=none] (19) at (-0.25, 0) {};
					\node [style=none] (20) at (2.75, 0) {};
					\node [style=none] (21) at (6.75, 0) {};
					\node [style=smallblue] (22) at (-10.25, 0) {};
					\node [style=smallblue] (23) at (7.75, 0) {};
					\node [style=none] (24) at (1.25, 0) {\dots};
					\node [style=none] (25) at (0.75, 0) {};
					\node [style=none] (26) at (1.75, 0) {};
					\node [style=none] (27) at (-7.25, -1.5) {$H$};
					\node [style=none] (28) at (-7.25, -2.75) {$H'$};
					\node [style=none] (29) at (-8.75, 1.25) {$F$};
				\end{pgfonlayer}
				\begin{pgfonlayer}{edgelayer}
					\draw [bend left=45, looseness=1.25] (4.center) to (5.center);
					\draw [bend right=60, looseness=1.50] (6.center) to (7.center);
					\draw [bend left=45, looseness=1.25] (8.center) to (9.center);
					\draw [bend right=60, looseness=1.50] (10.center) to (11.center);
					\draw [bend left=45, looseness=1.25] (12.center) to (13.center);
					\draw [bend right=60, looseness=1.50] (14.center) to (15.center);
					\draw [bend right=90, looseness=1.75] (2.center) to (3.center);
					\draw [bend left=90, looseness=1.75] (0.center) to (1.center);
					\draw (2.center) to (0.center);
					\draw (3.center) to (1.center);
					\draw [style=blue, bend left=90, looseness=0.75] (16.center) to (17.center);
					\draw [style=blue] (17.center) to (18.center);
					\draw [style=blue, bend left=90, looseness=0.75] (18.center) to (19.center);
					\draw [style=blue, bend left=90, looseness=0.75] (20.center) to (21.center);
					\draw [style=blue, bend right=90, looseness=0.75] (16.center) to (17.center);
					\draw [style=blue, bend right=90, looseness=0.75] (18.center) to (19.center);
					\draw [style=blue, bend right=90, looseness=0.75] (20.center) to (21.center);
					\draw [style=blue] (16.center) to (22);
					\draw [style=blue] (21.center) to (23);
					\draw [style=blue] (19.center) to (25.center);
					\draw [style=blue] (26.center) to (20.center);
				\end{pgfonlayer}
			\end{tikzpicture}
		\end{align}
		In this picture $\xi_F^H$ inside $H$ is complemented by the empty skein
		in $H'$.
		
		For the genus one stabilization in~\ref{hs2}, we need to change the Heegaard splitting by adding to $H$ a genus one handlebody $h$, and to $H'$ the complement $h^\complement$ of $h$ in $\mathbb{S}^3$.
		(This amounts to taking a connected sum of $M$ with $\mathbb{S}^3$, which is the reason why the overall manifold does not change.)
		This `new' Heegaard splitting can be described as follows:
		Let us consider a three-dimensional ball $B$ that intersects both $H$ and $H'$. The surface $\Sigma \cap B$ gives a decomposition of $B$ by two genus zero handlebodies that is depicted on the right below: 
		
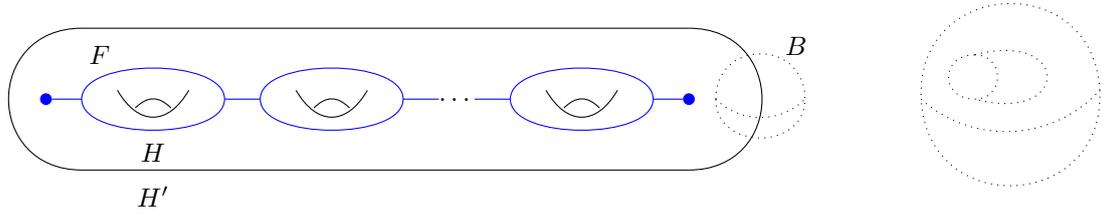
\begin{align}	\begin{tikzpicture}[scale=0.47]
				\begin{pgfonlayer}{nodelayer}
					\node [style=none] (0) at (13.25, 0.25) {};
					\node [style=none] (1) at (8.25, 0) {};
					\node [style=none] (2) at (9.75, 1.25) {};
					\node [style=none] (3) at (9.75, 0) {};
					\node [style=none] (4) at (1.75, 2) {};
					\node [style=none] (5) at (1.75, -2) {};
					\node [style=none] (6) at (-15.25, 2) {};
					\node [style=none] (7) at (-15.25, -2) {};
					\node [style=none] (8) at (-13.75, -0.25) {};
					\node [style=none] (9) at (-12.75, -0.25) {};
					\node [style=none] (10) at (-14.25, 0.25) {};
					\node [style=none] (11) at (-12.25, 0.25) {};
					\node [style=none] (12) at (-8.75, -0.25) {};
					\node [style=none] (13) at (-7.75, -0.25) {};
					\node [style=none] (14) at (-9.25, 0.25) {};
					\node [style=none] (15) at (-7.25, 0.25) {};
					\node [style=none] (16) at (-1.75, -0.25) {};
					\node [style=none] (17) at (-0.75, -0.25) {};
					\node [style=none] (18) at (-2.25, 0.25) {};
					\node [style=none] (19) at (-0.25, 0.25) {};
					\node [style=none] (20) at (-15.25, 0) {};
					\node [style=none] (21) at (-11.25, 0) {};
					\node [style=none] (22) at (-10.25, 0) {};
					\node [style=none] (23) at (-6.25, 0) {};
					\node [style=none] (24) at (-3.25, 0) {};
					\node [style=none] (25) at (0.75, 0) {};
					\node [style=smallblue] (26) at (-16.25, 0) {};
					\node [style=smallblue] (27) at (1.75, 0) {};
					\node [style=none] (28) at (-4.75, 0) {\dots};
					\node [style=none] (29) at (-5.25, 0) {};
					\node [style=none] (30) at (-4.25, 0) {};
					\node [style=none] (31) at (-13.25, -1.5) {$H$};
					\node [style=none] (32) at (-13.25, -2.75) {$H'$};
					\node [style=none] (33) at (-14.75, 1.25) {$F$};
					\node [style=none] (34) at (2.5, 0) {};
					\node [style=none] (35) at (5, 0) {};
					\node [style=none] (36) at (4.75, 1.5) {$B$};
				\end{pgfonlayer}
				\begin{pgfonlayer}{edgelayer}
					\draw [style=dashed, bend left=90, looseness=1.75] (0.center) to (1.center);
					\draw [style=dashed, bend right=90, looseness=1.75] (0.center) to (1.center);
					\draw [style=dashed, bend right=45] (1.center) to (0.center);
					\draw [style=dashed, bend right=90, looseness=2.00] (2.center) to (3.center);
					\draw [style=dashed, bend left=90, looseness=1.75] (2.center) to (3.center);
					\draw [style=dashed, bend left=105, looseness=5.75] (2.center) to (3.center);
					\draw [bend left=45, looseness=1.25] (8.center) to (9.center);
					\draw [bend right=60, looseness=1.50] (10.center) to (11.center);
					\draw [bend left=45, looseness=1.25] (12.center) to (13.center);
					\draw [bend right=60, looseness=1.50] (14.center) to (15.center);
					\draw [bend left=45, looseness=1.25] (16.center) to (17.center);
					\draw [bend right=60, looseness=1.50] (18.center) to (19.center);
					\draw [bend right=90, looseness=1.75] (6.center) to (7.center);
					\draw [bend left=90, looseness=1.75] (4.center) to (5.center);
					\draw (6.center) to (4.center);
					\draw (7.center) to (5.center);
					\draw [style=blue, bend left=90, looseness=0.75] (20.center) to (21.center);
					\draw [style=blue] (21.center) to (22.center);
					\draw [style=blue, bend left=90, looseness=0.75] (22.center) to (23.center);
					\draw [style=blue, bend left=90, looseness=0.75] (24.center) to (25.center);
					\draw [style=blue, bend right=90, looseness=0.75] (20.center) to (21.center);
					\draw [style=blue, bend right=90, looseness=0.75] (22.center) to (23.center);
					\draw [style=blue, bend right=90, looseness=0.75] (24.center) to (25.center);
					\draw [style=blue] (20.center) to (26);
					\draw [style=blue] (25.center) to (27);
					\draw [style=blue] (23.center) to (29.center);
					\draw [style=blue] (30.center) to (24.center);
					\draw [style=dashed, bend left=90, looseness=1.75] (34.center) to (35.center);
					\draw [style=dashed, bend right=90, looseness=1.50] (34.center) to (35.center);
					\draw [style=dashed, bend right=45] (34.center) to (35.center);
				\end{pgfonlayer}
			\end{tikzpicture}	
		\end{align}
		If one considers the decomposition of $B$ by two genus one handlebodies (such a decomposition is described on the right below), one gets the Heegaard splitting of $M$ by $H' \# h^\complement$ and $H \# h$:
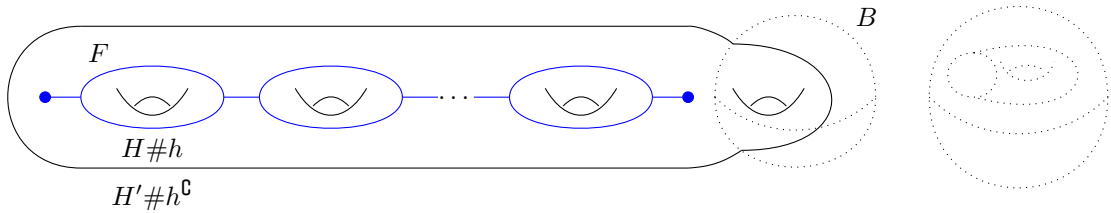
		\begin{align}	\begin{tikzpicture}[scale=0.47]
				\begin{pgfonlayer}{nodelayer}
					\node [style=none] (0) at (13.75, 0) {};
					\node [style=none] (1) at (8.75, 0) {};
					\node [style=none] (2) at (10, 1.25) {};
					\node [style=none] (3) at (10, 0) {};
					\node [style=none] (4) at (11, 0.75) {};
					\node [style=none] (5) at (12, 0.75) {};
					\node [style=none] (6) at (12.25, 1) {};
					\node [style=none] (7) at (10.75, 1) {};
					\node [style=none] (8) at (11.75, -1.25) {};
					\node [style=none] (9) at (2, 2) {};
					\node [style=none] (10) at (2, -2) {};
					\node [style=none] (11) at (-15, 2) {};
					\node [style=none] (12) at (-15, -2) {};
					\node [style=none] (13) at (-13.5, -0.25) {};
					\node [style=none] (14) at (-12.5, -0.25) {};
					\node [style=none] (15) at (-14, 0.25) {};
					\node [style=none] (16) at (-12, 0.25) {};
					\node [style=none] (17) at (-8.5, -0.25) {};
					\node [style=none] (18) at (-7.5, -0.25) {};
					\node [style=none] (19) at (-9, 0.25) {};
					\node [style=none] (20) at (-7, 0.25) {};
					\node [style=none] (21) at (-1.5, -0.25) {};
					\node [style=none] (22) at (-0.5, -0.25) {};
					\node [style=none] (23) at (-2, 0.25) {};
					\node [style=none] (24) at (0, 0.25) {};
					\node [style=none] (25) at (-15, 0) {};
					\node [style=none] (26) at (-11, 0) {};
					\node [style=none] (27) at (-10, 0) {};
					\node [style=none] (28) at (-6, 0) {};
					\node [style=none] (29) at (-3, 0) {};
					\node [style=none] (30) at (1, 0) {};
					\node [style=smallblue] (31) at (-16, 0) {};
					\node [style=smallblue] (32) at (2, 0) {};
					\node [style=none] (33) at (-4.5, 0) {\dots};
					\node [style=none] (34) at (-5, 0) {};
					\node [style=none] (35) at (-4, 0) {};
					\node [style=none] (36) at (-13, -1.5) {$H \# h$};
					\node [style=none] (37) at (-13, -2.75) {$H' \# h^\complement$};
					\node [style=none] (38) at (-14.5, 1.25) {$F$};
					\node [style=none] (39) at (2.75, 0) {};
					\node [style=none] (40) at (7.25, 0) {};
					\node [style=none] (41) at (3.25, 1.5) {};
					\node [style=none] (42) at (3.5, -1.5) {};
					\node [style=none] (43) at (3.75, -0.25) {};
					\node [style=none] (44) at (4.75, -0.25) {};
					\node [style=none] (45) at (3.25, 0.25) {};
					\node [style=none] (46) at (5.25, 0.25) {};
					\node [style=none] (47) at (7, 2.25) {$B$};
				\end{pgfonlayer}
				\begin{pgfonlayer}{edgelayer}
					\draw [style=dashed, bend left=90, looseness=1.75] (0.center) to (1.center);
					\draw [style=dashed, bend right=90, looseness=1.75] (0.center) to (1.center);
					\draw [style=dashed, bend right=45] (1.center) to (0.center);
					\draw [style=dashed, bend right=90, looseness=2.00] (2.center) to (3.center);
					\draw [style=dashed, bend left=90, looseness=1.75] (2.center) to (3.center);
					\draw [style=dashed, bend left=105, looseness=8.25] (2.center) to (3.center);
					\draw [style=dashed, bend left] (4.center) to (5.center);
					\draw [style=dashed, bend right=75, looseness=1.25] (7.center) to (6.center);
					\draw [bend left=45, looseness=1.25] (13.center) to (14.center);
					\draw [bend right=60, looseness=1.50] (15.center) to (16.center);
					\draw [bend left=45, looseness=1.25] (17.center) to (18.center);
					\draw [bend right=60, looseness=1.50] (19.center) to (20.center);
					\draw [bend left=45, looseness=1.25] (21.center) to (22.center);
					\draw [bend right=60, looseness=1.50] (23.center) to (24.center);
					\draw [bend right=90, looseness=1.75] (11.center) to (12.center);
					\draw (11.center) to (9.center);
					\draw (12.center) to (10.center);
					\draw [style=blue, bend left=90, looseness=0.75] (25.center) to (26.center);
					\draw [style=blue] (26.center) to (27.center);
					\draw [style=blue, bend left=90, looseness=0.75] (27.center) to (28.center);
					\draw [style=blue, bend left=90, looseness=0.75] (29.center) to (30.center);
					\draw [style=blue, bend right=90, looseness=0.75] (25.center) to (26.center);
					\draw [style=blue, bend right=90, looseness=0.75] (27.center) to (28.center);
					\draw [style=blue, bend right=90, looseness=0.75] (29.center) to (30.center);
					\draw [style=blue] (25.center) to (31);
					\draw [style=blue] (30.center) to (32);
					\draw [style=blue] (28.center) to (34.center);
					\draw [style=blue] (35.center) to (29.center);
					\draw [style=dashed, bend left=90, looseness=1.75] (39.center) to (40.center);
					\draw [style=dashed, bend right=90, looseness=1.50] (39.center) to (40.center);
					\draw [style=dashed, bend right=45] (39.center) to (40.center);
					\draw [bend right=15] (10.center) to (42.center);
					\draw [bend left, looseness=0.50] (9.center) to (41.center);
					\draw [in=0, out=0, looseness=3.00] (41.center) to (42.center);
					\draw [bend left=45, looseness=1.25] (43.center) to (44.center);
					\draw [bend right=60, looseness=1.50] (45.center) to (46.center);
				\end{pgfonlayer}
			\end{tikzpicture}			
		\end{align}
Now we will use the fact that $F$ is special, in particular we have $\mu \delta = \text{id}_F$. One can pull $v_M^F$ inside $B$, this obviously does not change the skein in $M$ because of the isotopy invariance of skeins:
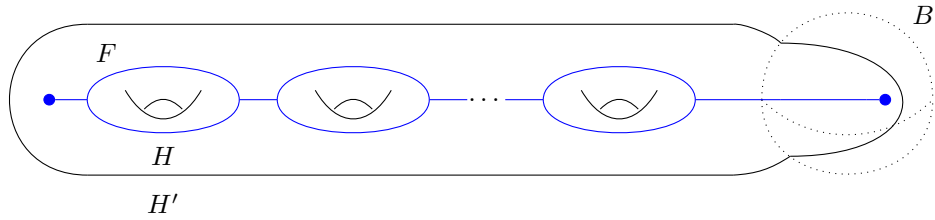
\begin{align}	\begin{tikzpicture}[scale=0.5]
		\begin{pgfonlayer}{nodelayer}
			\node [style=none] (0) at (7, 2) {};
			\node [style=none] (1) at (7, -2) {};
			\node [style=none] (2) at (-10, 2) {};
			\node [style=none] (3) at (-10, -2) {};
			\node [style=none] (4) at (-8.5, -0.25) {};
			\node [style=none] (5) at (-7.5, -0.25) {};
			\node [style=none] (6) at (-9, 0.25) {};
			\node [style=none] (7) at (-7, 0.25) {};
			\node [style=none] (8) at (-3.5, -0.25) {};
			\node [style=none] (9) at (-2.5, -0.25) {};
			\node [style=none] (10) at (-4, 0.25) {};
			\node [style=none] (11) at (-2, 0.25) {};
			\node [style=none] (12) at (3.5, -0.25) {};
			\node [style=none] (13) at (4.5, -0.25) {};
			\node [style=none] (14) at (3, 0.25) {};
			\node [style=none] (15) at (5, 0.25) {};
			\node [style=none] (16) at (-10, 0) {};
			\node [style=none] (17) at (-6, 0) {};
			\node [style=none] (18) at (-5, 0) {};
			\node [style=none] (19) at (-1, 0) {};
			\node [style=none] (20) at (2, 0) {};
			\node [style=none] (21) at (6, 0) {};
			\node [style=smallblue] (22) at (-11, 0) {};
			\node [style=smallblue] (23) at (11, 0) {};
			\node [style=none] (24) at (0.5, 0) {\dots};
			\node [style=none] (25) at (0, 0) {};
			\node [style=none] (26) at (1, 0) {};
			\node [style=none] (27) at (-8, -1.5) {$H$};
			\node [style=none] (28) at (-8, -2.75) {$H'$};
			\node [style=none] (29) at (-9.5, 1.25) {$F$};
			\node [style=none] (30) at (7.75, 0) {};
			\node [style=none] (31) at (12.25, 0) {};
			\node [style=none] (32) at (8.25, 1.5) {};
			\node [style=none] (33) at (8.5, -1.5) {};
			\node [style=none] (38) at (8.25, 0) {};
			\node [style=none] (39) at (10.5, 0) {};
			\node [style=none] (40) at (11.75, 2.5) {};
			\node [style=none] (41) at (11.75, 2.5) {};
			\node [style=none] (42) at (12, 2.25) {$B$};
		\end{pgfonlayer}
		\begin{pgfonlayer}{edgelayer}
			\draw [bend left=45, looseness=1.25] (4.center) to (5.center);
			\draw [bend right=60, looseness=1.50] (6.center) to (7.center);
			\draw [bend left=45, looseness=1.25] (8.center) to (9.center);
			\draw [bend right=60, looseness=1.50] (10.center) to (11.center);
			\draw [bend left=45, looseness=1.25] (12.center) to (13.center);
			\draw [bend right=60, looseness=1.50] (14.center) to (15.center);
			\draw [bend right=90, looseness=1.75] (2.center) to (3.center);
			\draw (2.center) to (0.center);
			\draw (3.center) to (1.center);
			\draw [style=blue, bend left=90, looseness=0.75] (16.center) to (17.center);
			\draw [style=blue] (17.center) to (18.center);
			\draw [style=blue, bend left=90, looseness=0.75] (18.center) to (19.center);
			\draw [style=blue, bend left=90, looseness=0.75] (20.center) to (21.center);
			\draw [style=blue, bend right=90, looseness=0.75] (16.center) to (17.center);
			\draw [style=blue, bend right=90, looseness=0.75] (18.center) to (19.center);
			\draw [style=blue, bend right=90, looseness=0.75] (20.center) to (21.center);
			\draw [style=blue] (16.center) to (22);
			\draw [style=blue] (19.center) to (25.center);
			\draw [style=blue] (26.center) to (20.center);
			\draw [style=dashed, bend left=90, looseness=1.75] (30.center) to (31.center);
			\draw [style=dashed, bend right=90, looseness=1.50] (30.center) to (31.center);
			\draw [style=dashed, bend right=45] (30.center) to (31.center);
			\draw [bend right=15] (1.center) to (33.center);
			\draw [bend left, looseness=0.50] (0.center) to (32.center);
			\draw [in=0, out=0, looseness=3.50] (32.center) to (33.center);
			\draw [style=blue] (21.center) to (38.center);
			\draw [style=blue] (38.center) to (39.center);
			\draw [style=blue] (39.center) to (23);
		\end{pgfonlayer}
	\end{tikzpicture}
\end{align}
As $F$ is special, this vector is the same as the following one:
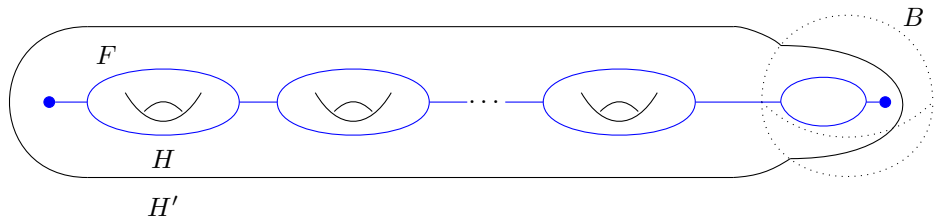
\begin{align}	\begin{tikzpicture}[scale=0.5]
		\begin{pgfonlayer}{nodelayer}
			\node [style=none] (0) at (7, 2) {};
			\node [style=none] (1) at (7, -2) {};
			\node [style=none] (2) at (-10, 2) {};
			\node [style=none] (3) at (-10, -2) {};
			\node [style=none] (4) at (-8.5, -0.25) {};
			\node [style=none] (5) at (-7.5, -0.25) {};
			\node [style=none] (6) at (-9, 0.25) {};
			\node [style=none] (7) at (-7, 0.25) {};
			\node [style=none] (8) at (-3.5, -0.25) {};
			\node [style=none] (9) at (-2.5, -0.25) {};
			\node [style=none] (10) at (-4, 0.25) {};
			\node [style=none] (11) at (-2, 0.25) {};
			\node [style=none] (12) at (3.5, -0.25) {};
			\node [style=none] (13) at (4.5, -0.25) {};
			\node [style=none] (14) at (3, 0.25) {};
			\node [style=none] (15) at (5, 0.25) {};
			\node [style=none] (16) at (-10, 0) {};
			\node [style=none] (17) at (-6, 0) {};
			\node [style=none] (18) at (-5, 0) {};
			\node [style=none] (19) at (-1, 0) {};
			\node [style=none] (20) at (2, 0) {};
			\node [style=none] (21) at (6, 0) {};
			\node [style=smallblue] (22) at (-11, 0) {};
			\node [style=smallblue] (23) at (11, 0) {};
			\node [style=none] (24) at (0.5, 0) {\dots};
			\node [style=none] (25) at (0, 0) {};
			\node [style=none] (26) at (1, 0) {};
			\node [style=none] (27) at (-8, -1.5) {$H$};
			\node [style=none] (28) at (-8, -2.75) {$H'$};
			\node [style=none] (29) at (-9.5, 1.25) {$F$};
			\node [style=none] (30) at (7.75, 0) {};
			\node [style=none] (31) at (12.25, 0) {};
			\node [style=none] (32) at (8.25, 1.5) {};
			\node [style=none] (33) at (8.5, -1.5) {};
			\node [style=none] (38) at (8.25, 0) {};
			\node [style=none] (39) at (10.5, 0) {};
			\node [style=none] (40) at (11.75, 2.25) {$B$};
		\end{pgfonlayer}
		\begin{pgfonlayer}{edgelayer}
			\draw [bend left=45, looseness=1.25] (4.center) to (5.center);
			\draw [bend right=60, looseness=1.50] (6.center) to (7.center);
			\draw [bend left=45, looseness=1.25] (8.center) to (9.center);
			\draw [bend right=60, looseness=1.50] (10.center) to (11.center);
			\draw [bend left=45, looseness=1.25] (12.center) to (13.center);
			\draw [bend right=60, looseness=1.50] (14.center) to (15.center);
			\draw [bend right=90, looseness=1.75] (2.center) to (3.center);
			\draw (2.center) to (0.center);
			\draw (3.center) to (1.center);
			\draw [style=blue, bend left=90, looseness=0.75] (16.center) to (17.center);
			\draw [style=blue] (17.center) to (18.center);
			\draw [style=blue, bend left=90, looseness=0.75] (18.center) to (19.center);
			\draw [style=blue, bend left=90, looseness=0.75] (20.center) to (21.center);
			\draw [style=blue, bend right=90, looseness=0.75] (16.center) to (17.center);
			\draw [style=blue, bend right=90, looseness=0.75] (18.center) to (19.center);
			\draw [style=blue, bend right=90, looseness=0.75] (20.center) to (21.center);
			\draw [style=blue] (16.center) to (22);
			\draw [style=blue] (19.center) to (25.center);
			\draw [style=blue] (26.center) to (20.center);
			\draw [style=dashed, bend left=90, looseness=1.75] (30.center) to (31.center);
			\draw [style=dashed, bend right=90, looseness=1.50] (30.center) to (31.center);
			\draw [style=dashed, bend right=45] (30.center) to (31.center);
			\draw [bend right=15] (1.center) to (33.center);
			\draw [bend left, looseness=0.50] (0.center) to (32.center);
			\draw [in=0, out=0, looseness=3.50] (32.center) to (33.center);
			\draw [style=blue] (21.center) to (38.center);
			\draw [style=blue, bend left=90] (38.center) to (39.center);
			\draw [style=blue] (39.center) to (23);
			\draw [style=blue, bend right=75] (38.center) to (39.center);
		\end{pgfonlayer}
	\end{tikzpicture}	
\end{align}
Now, by the genus one decomposition of $B$, this skein in $M$ can be seen as 
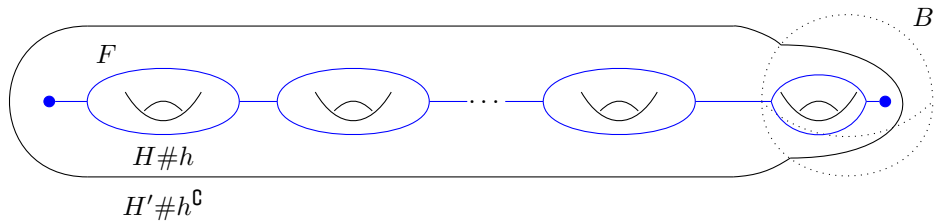
\begin{align}	\begin{tikzpicture}[scale=0.5]
		\begin{pgfonlayer}{nodelayer}
			\node [style=none] (0) at (7, 2) {};
			\node [style=none] (1) at (7, -2) {};
			\node [style=none] (2) at (-10, 2) {};
			\node [style=none] (3) at (-10, -2) {};
			\node [style=none] (4) at (-8.5, -0.25) {};
			\node [style=none] (5) at (-7.5, -0.25) {};
			\node [style=none] (6) at (-9, 0.25) {};
			\node [style=none] (7) at (-7, 0.25) {};
			\node [style=none] (8) at (-3.5, -0.25) {};
			\node [style=none] (9) at (-2.5, -0.25) {};
			\node [style=none] (10) at (-4, 0.25) {};
			\node [style=none] (11) at (-2, 0.25) {};
			\node [style=none] (12) at (3.5, -0.25) {};
			\node [style=none] (13) at (4.5, -0.25) {};
			\node [style=none] (14) at (3, 0.25) {};
			\node [style=none] (15) at (5, 0.25) {};
			\node [style=none] (16) at (-10, 0) {};
			\node [style=none] (17) at (-6, 0) {};
			\node [style=none] (18) at (-5, 0) {};
			\node [style=none] (19) at (-1, 0) {};
			\node [style=none] (20) at (2, 0) {};
			\node [style=none] (21) at (6, 0) {};
			\node [style=smallblue] (22) at (-11, 0) {};
			\node [style=smallblue] (23) at (11, 0) {};
			\node [style=none] (24) at (0.5, 0) {\dots};
			\node [style=none] (25) at (0, 0) {};
			\node [style=none] (26) at (1, 0) {};
			\node [style=none] (27) at (-8, -1.5) {$H \# h$};
			\node [style=none] (28) at (-8, -2.75) {$H' \# h^\complement$};
			\node [style=none] (29) at (-9.5, 1.25) {$F$};
			\node [style=none] (30) at (7.75, 0) {};
			\node [style=none] (31) at (12.25, 0) {};
			\node [style=none] (32) at (8.25, 1.5) {};
			\node [style=none] (33) at (8.5, -1.5) {};
			\node [style=none] (34) at (8.75, -0.25) {};
			\node [style=none] (35) at (9.75, -0.25) {};
			\node [style=none] (36) at (8.25, 0.25) {};
			\node [style=none] (37) at (10.25, 0.25) {};
			\node [style=none] (38) at (8, 0) {};
			\node [style=none] (39) at (10.5, 0) {};
			\node [style=none] (40) at (12, 2.25) {$B$};
		\end{pgfonlayer}
		\begin{pgfonlayer}{edgelayer}
			\draw [bend left=45, looseness=1.25] (4.center) to (5.center);
			\draw [bend right=60, looseness=1.50] (6.center) to (7.center);
			\draw [bend left=45, looseness=1.25] (8.center) to (9.center);
			\draw [bend right=60, looseness=1.50] (10.center) to (11.center);
			\draw [bend left=45, looseness=1.25] (12.center) to (13.center);
			\draw [bend right=60, looseness=1.50] (14.center) to (15.center);
			\draw [bend right=90, looseness=1.75] (2.center) to (3.center);
			\draw (2.center) to (0.center);
			\draw (3.center) to (1.center);
			\draw [style=blue, bend left=90, looseness=0.75] (16.center) to (17.center);
			\draw [style=blue] (17.center) to (18.center);
			\draw [style=blue, bend left=90, looseness=0.75] (18.center) to (19.center);
			\draw [style=blue, bend left=90, looseness=0.75] (20.center) to (21.center);
			\draw [style=blue, bend right=90, looseness=0.75] (16.center) to (17.center);
			\draw [style=blue, bend right=90, looseness=0.75] (18.center) to (19.center);
			\draw [style=blue, bend right=90, looseness=0.75] (20.center) to (21.center);
			\draw [style=blue] (16.center) to (22);
			\draw [style=blue] (19.center) to (25.center);
			\draw [style=blue] (26.center) to (20.center);
			\draw [style=dashed, bend left=90, looseness=1.75] (30.center) to (31.center);
			\draw [style=dashed, bend right=90, looseness=1.50] (30.center) to (31.center);
			\draw [style=dashed, bend right=45] (30.center) to (31.center);
			\draw [bend right=15] (1.center) to (33.center);
			\draw [bend left, looseness=0.50] (0.center) to (32.center);
			\draw [in=0, out=0, looseness=3.50] (32.center) to (33.center);
			\draw [bend left=45, looseness=1.25] (34.center) to (35.center);
			\draw [bend right=60, looseness=1.50] (36.center) to (37.center);
			\draw [style=blue] (21.center) to (38.center);
			\draw [style=blue, bend left=75] (38.center) to (39.center);
			\draw [style=blue, bend right=75, looseness=1.25] (38.center) to (39.center);
			\draw [style=blue] (39.center) to (23);
		\end{pgfonlayer}
	\end{tikzpicture}
\end{align}
This is, by definition, the vector that we assign to $H \# h$, see~\eqref{eqnstabilization_elaborate}. This proves~\refeq{hs2} and we conclude that $v_M^F$ does not depend on the choice of the Heegaard splitting. 

The diffeomorphism invariance of $v_M^F$ is already implicit in the independence of the Heegaard splitting. Let us spell this out: Let $\psi:M \cong M$ be a diffeomorphism. Choose a Heegaard splitting $H' \cup_{\Sigma} H \cong M$. Then, $\psi$ gives another Heegaard splitting of $M$, however this new Heegaard splitting is equivalent to the one that we chose, by the previous three moves~\ref{hs3},~\ref{hs1} and~\ref{hs2} \cite[Section~5.8]{BK}. We already showed that the vector $v_M^F \in \text{sk}_{\mathcal{A}}(M)$ is preserved by these three moves. Therefore, $v_M^F$ is invariant with respect to the diffeomorphisms of $M$.
\end{proof}

\section{Connection to Reshetikhin-Turaev and Crane-Yetter topological field theories}

The Reshetikhin-Turaev construction of a three-dimensional topological field theory \cite{RT1,RT2,Tur} takes as an algebraic input a modular fusion category $\mathcal{A}$, i.e.\
a ribbon fusion category with non-degenerate braiding. Non-degeneracy of the braiding means that every object that trivially double braids with all other objects is isomorphic to a direct sum of the monoidal unit. This topological field theory has a so-called \emph{framing anomaly}, meaning that it is not a symmetric monoidal functor in the usual sense but with a composition that is well-defined up to a scalar. The Reshetikhin-Turaev topological field theory can be seen as a boundary condition in the sense of \cite{Walker,FT,Ben} to the four-dimensional Crane-Yetter topological field theory $\mathcal{Z}^{\text{CY}}$ \cite{CY,CKY,CGHP}. The latter takes a ribbon fusion category $\mathcal{A}$ as input. The value of the Crane-Yetter topological field theory is given by skein-theoretical constructions. In particular, the vector space that is associated to a closed three-dimensional manifold $M$ is the skein module $\text{sk}_{\mathcal{A}}(M)$. The vector $v_M^F \in \text{sk}_{\mathcal{A}}(M)$ of Theorem~\ref{thmain2}, being a diffeomorphism invariant of $M$, can be seen as a boundary condition for $\mathcal{Z}^{\text{CY}}$ that, here, is just defined in dimension three, that is not necessarily given by the empty skein if $F \neq I$.

If $\mathcal{A}$ is moreover modular, the Reshetikhin-Turaev topological field theory can be recovered as an anomalous theory associated to $\mathcal{Z}^{\text{CY}}$ with the boundary condition given by the empty skein, see \cite[Section 6]{Ben} for the precise statement. In particular, the Reshetikhin-Turaev invariant of a closed three-manifold $M$ can be seen as a map \begin{align}
k \rightarrow \text{sk}_{\mathcal{A}}(M) \xrightarrow{\mathcal{Z}_{\mathcal{A}}^{\text{CY}}(W)} k \ ,   \label{eqnnumber}
\end{align}
where the first map points to the empty skein in $M$ and $W$ is a four-dimensional manifold with $\partial W = M$. The vector $v_M^F$ in Theorem~\ref{thmain2} for $F=I$ is the empty skein. Therefore, we obtain the following consequence:
\begin{corollary}\label{son}
Let $W$ be a four-dimensional manifold with $\partial W = M$. Then, the number $\mathcal{Z}^{\text{CY}}(W)(v_M^I)$ is the Reshetikhin-Turaev invariant of $M$.
\end{corollary}

The production of a number using~\eqref{eqnnumber} works of course also for a special symmetric commutative Frobenius algebra $F\neq I$. It then yields a manifold invariant
$\mathcal{Z}^{\text{CY}}(W)(v_M^F)$. To see examples of special symmetric commutative Frobenius algebras $F \neq I$, we refer to \cite{kirillovostrik,Dav}.

\medskip
\small
\newcommand{\etalchar}[1]{$^{#1}$}

\noindent \textsc{Université Bourgogne Europe, CNRS, IMB UMR 5584, F-21000 Dijon, France}

\end{document}